\documentclass{ifacconf}

\usepackage{graphicx}      
\usepackage{natbib}        
\usepackage{cite}
\usepackage{amsmath,amssymb,amsfonts}
\usepackage{algorithmic}
\usepackage{algorithm,algorithmic}
\usepackage{mathtools,mathbbol,dsfont}
\usepackage{bm}
\usepackage{multirow}
\usepackage{threeparttable}
\usepackage{textcomp}
\usepackage[all]{xy}
\usepackage{chemarrow}
\usepackage{color}
\usepackage{makecell}
\usepackage{graphicx}
\usepackage{textcomp}
\usepackage{array} 
\usepackage{longtable}
\usepackage{booktabs}
\usepackage{float}
\usepackage{mathrsfs,array}
\makeatletter

\newcommand{\Rmnum}[1]{\expandafter\@slowromancap\romannumeral #1@}
\makeatother
\begin{document}
\begin{frontmatter}

\title{On the Convergence Rate Lower Bound of Biochemical Computational Modules } 



\author[First]{Yuzhen Fan} 
\author[Second]{Chuanhou Gao}
\author[First]{Shibo He}
\author[First]{Jiming Chen}

\address[First]{College of Control Science and Engineering, 
   Hangzhou, 310027, China (e-mail: yzfan@zju.edu.cn, s18he@zju.edu.cn, cjm@zju.edu.cn).}
\address[Second]{School of Mathematical Sciences, Hangzhou, 310058, China (e-mail: gaochou@zju.edu.cn)}


\begin{abstract}                
Biochemical reaction networks have become a central theoretical framework for implementing molecular computation. A key challenge is finite time computational accuracy, as computation outputs are encoded in limiting steady states (LSSs) of species concentrations while practical implementations operate for only finite time. This work proposes a concise characterization of convergence rate for biochemical computational modules with multiple output species, and rigorously establishes it as being bounded by the eigenvalue with largest (least negative) real part of the Jacobian matrix. Two numerical examples illustrate how the theoretical lower bound shapes the convergence rate range and reveals its dependence on reaction rate constants. This formulation enables systematic evaluation of biochemical computation speed and provides a practical design measure for constructing high-accuracy and error-controlled biochemical computational modules. 
\end{abstract}


\end{frontmatter}

\section{Introduction}
Interactions among biomolecules exhibit complicated dynamic behaviors that enable living cells to process multiple input signals and make decisions accordingly, which allow biological organisms to adapt to fluctuating environments and maintain functioning. Driven by both synthetic biology applications and the exploration of biological learning mechanisms, increasing theoretical \citep{vasic2022programming,anderson2021reaction,fan2025automatic,arredondo2022supervised,moorman2019dynamical,singh2019reaction} and experimental \citep{cherry2025supervised, chen2024synthetic,okumura2022nonlinear} efforts have been devoted to constructing engineered biological circuits capable of learning based on existing computational paradigms, especially artificial neural networks. Those facilitate the execution of complex computational and information processing tasks at the cellular and molecular levels.

Chemical reaction network (CRN) represents a widely recognized mathematical model for describing biomolecular interactions and serves as a theoretical bridge for the system level design of biological molecular circuits. Since CRNs can be implemented via DNA strand displacement reactions, there is growing interest in exploring CRNs as computational machines \citep{vasic2020crn++}, along with development in mathematical theory that are crucial for ensuring reliable design method and advancing refinement of theoretical design \citep{anderson2021reaction,anderson2025chemical,chalk2019composable,1fan2025automatic2}.
A fundamental theoretical challenge in regarding CRNs as a computational machine lies in the finite time computational accuracy problem\citep{anderson2025chemical,1fan2025automatic2}. Specifically, when using CRNs to perform a computational task, the output value is conventionally designed as the LSSs of some species that the system approaches asymptotically \citep{anderson2021reaction}. Since any practical implementation requires finite time operation of the system, it is important to bound the error, i.e., deviations between the actual concentrations recorded at the finite time and the desired computational values encoded by LSSs. Such errors can be propagated and amplified in large-scale designs, degrading the performance of the biochemical computational systems \citep{1fan2025automatic2}. 

Extending the recording time to improve accuracy is unsustainable as it compromises practicality and introduces additional issues. Therefore, for implementing given computational tasks while maintaining the desired accuracy using CRNs, it is essential to design the CRN system (structure and parameters) with faster convergence rate. This relies on a concise measure of the convergence rate, or equivalently the error decay rate, which is regarded as computation speed when using CRNs. Anderson first provided a quantitative definition on convergence rates for real-valued functions, which facilitates to evaluate biochemical computation modules with single output \citep{anderson2025chemical}. However, for multiple output cases with coupled dynamics (e.g. biochemical neural networks) or more complex single output designs for nonlinear operations such as computing Sigmoid function, it is difficult to directly evaluate computation speed through the definition. In this work, we present a characterization of the convergence rate lower bound for biochemical computational modules that is formulated for one whole mass action system (CRNs equipped with mass action kinetics) and linked to by the eigenvalue with largest (least negative) real part of its Jacobian matrix. Our results allow the biochemical computation speed to be estimated without explicit solution and provide a quantitative metric for error-controlled design, where higher computation speed directly implies higher computational accuracy within a given operation time.

The remainder of this article is organized as follows. Section \ref{sec:pre} covers preliminaries about CRNs and computation using CRNs. Section \ref{sec:convergence rate} develops the characterization of the convergence rate lower bound for both linear and nonlinear mass action systems with rigorous proofs. Two numerical examples are presented in Section \ref{sec:example} to illustrate our theoretical results.

\section{Preliminaries}\label{sec:pre}
In this section, we briefly introduce the definition of chemical reaction networks and the computation using CRNs as a biochemical programming language.

\subsection{Chemical Reaction Network}

Consider a CRN with $n$ species that interact with each other through $r$ reactions
\begin{defn}[CRN]
A CRN consists of three finite sets: \textit{species set} $\mathcal{S}=\{X_{1}, \ldots,X_{n}\}$ that denotes the subjects participating in the reactions; \textit{complex set} $\mathcal{C}=\bigcup_{j=1}^{r}\{C_{j},C'_{j}\}$ where each complex is a nonnegative integer linear combination of species, $\sum_{i=1}^{n} a_{i} X_{i}$, with stoichiometric coefficient $a_i \in \mathbb{Z}_{\geq 0}$ of $X_i$;
\textit{reaction set} $\mathcal{R}=\{C_j\to C'_{j}: j=1,...,r\}$ satisfying that $\forall C_{j}\to C'_{j}\in \mathcal{R}$, $C_{.j}\neq C'_{.j}$, and $\forall C_{.j}\in \mathcal{C}$, $\exists C'_{.j}\in\mathcal{C}$ supporting either $C_{.j}\to C'_{.j}$ or $C'_{.j}\to C_{.j}$. We use the triple $(\mathcal{S},\mathcal{C},\mathcal{R})$ to represent a CRN.


\end{defn}

Then, the $j$th reaction is written as
\begin{equation}
	\sum_{i=1}^{n} v_{ij} X_{i} \longrightarrow \sum_{i=1}^{n} v^{\prime}_{ij} X_{i},
	\label{eq:1}
\end{equation} 
where $\sum_{i=1}^{n} v_{ij} X_{i}$ and $\sum_{i=1}^{n} v^{\prime}_{ij} X_{i}$ represent the reactant and product complex of the $j$th reaction, respectively. Further, we define the \textit{reaction vectors} of the $j$th reaction by $v'_{.j}-v_{.j}$.
The dynamics of $(\mathcal{S},\mathcal{C},\mathcal{R})$ that captures the concentration change of each species, labeled by $x\in\mathbb{R}^n_{\geq 0}$, if a continuous vector-valued function $\mathscr{K} \colon \mathbb{R}^n_{\geq 0} \to \mathbb{R}^r_{\geq 0}$ is defined to evaluate the reaction rates, written as
\begin{equation}
    \frac{dx(t)}{dt} = \Gamma \mathscr{K}(x),~~x \in \mathbb{R}^n_{\geq0}.
    \label{eq:2}
\end{equation}
where $\Gamma\in\mathbb{Z}^{n\times r}$ is the \textit{stoichiometric matrix} with the $j$th column $\Gamma_{\cdot j}=v'_{.j}-v_{.j}$. 
\textit{Mass action kinetics} is usually used to evaluate the reaction rate, which induces the rate of the $j$th reaction as
\begin{equation}\label{mak}
\mathscr{K}_j(x)=  k_j x^{v.j} \triangleq k_j\prod_{i=1}^{n} x_{i}^{v_{ij}},
\end{equation}
where $k_j>0$ represents the reaction rate constant. 
The CRN system endowed with mass action kinetics offers polynomial ordinary differential equations, often termed as \textit{mass action system} (MAS) and denoted by the quad $(\mathcal{S},\mathcal{C},\mathcal{R},k)$ with $k=(k_1,..., k_r)^{\top}$. With $k_j$ specified, the reaction graph \eqref{eq:1} can be rewritten with $k_j$ alongside the reaction arrow.




We further give the definitions of equilibrium and exponential convergence. 
\begin{defn}[Equilibrium]

For a MAS $(\mathcal{S},\mathcal{C},\mathcal{R},k)$ governed by (\ref{eq:2}) plus (\ref{mak}), if a constant vector $\bar{x} \in 
\mathbb{R}^n_{\geq 0}$ satisfies $\Gamma \mathscr{K}(\bar{x})=0$, then $\bar{x}$ is called the \textit{nonnegative equilibrium point}. If for any $ x_{\mathcal{A}}\in\mathcal{A} \subset \mathbb{R}^n_{\geq 0}$ there is $\Gamma \mathscr{K}(x_{\mathcal{A}}) =0 $, then the set $\mathcal{A}$ is called an equilibrium set.
\end{defn}

\begin{defn}[Exponential Convergence] Consider a MAS $(\mathcal{S},\mathcal{C},\mathcal{R},k)$ governed by (\ref{eq:2}) plus (\ref{mak}) and admitting an equilibrium set $\mathcal{A}$. The solution $x(t)$ of this MAS \textit{converges exponentially} to $\mathcal{A}$ if there are constants $M,\gamma \in \mathbb{R}_{>0}$ supporting $\inf_{y \in \mathcal{A}} \Vert x(t)-y \Vert \leq M e^{-\gamma t}$ for all $t\geq 0$. When the system admits an isolated equilibrium, i.e., $\mathcal{A} = \{\bar{x}\}$, \textit{exponential convergence} is ensured if there exist $M,\gamma \in \mathbb{R}_{>0}$ satisfying $\Vert x(t) - \bar{x} \Vert \leq M e^{-\gamma t}$.
\end{defn}

\subsection{Implementing Calculation Using CRNs}
Various calculations could be realized by CRNs. It implies that each variable (a fixed constant in calculation or information processing) needs to be represented by some counterpart in CRNs. Thus, considering the CRN $(\mathcal{S},\mathcal{C},\mathcal{R})$ govern by  (\ref{eq:2}) and (\ref{mak}) with $\mathcal{S} = \mathcal{X}_1 \bigcup \mathcal{X}_2$, where $\mathcal{X}_1$ denotes the input species set and $\mathcal{X}_2$ refers to the output species set, the dynamical equations are written as 
\begin{equation}\label{eq:bcm}
\left\{
\begin{aligned}
\dot{x}_1(t) & = f_1(x_1, x_2), \\
\dot{x}_2(t) & = f_2(x_1, x_2)
\end{aligned}
\right.
\end{equation}
with $x_1 (t) \in \mathbb{R}^{n_1}_{\geq 0}, x_2 (t) \in \mathbb{R}^{n-n_1}_{\geq 0}$ indicate the concentration of input species and output species, respectively. When implementing the calculation of the function $g(s) \in \mathbb{R}^{n-n_1}$ with $s \in \mathbb{R}^{n_1}$ using CRNs, it prompts to use $x_1(0)$ to encode values of the independent variable $s$, while the LSSs of output species, i.e., 
\begin{equation} \label{def:lss}
  \mathscr{L}[x(t)]= (\lim_{t \to \infty} x_2(t))_{\mathcal{X}_2 \subset \mathcal{S}}
\end{equation}
to represent the corresponding dependent variables in any calculation. We refer to the CRN that can accomplish the desired computational task as a \textit{biochemical computational module}. Note that \eqref{def:lss} includes globally/locally asymptotically stable equilibrium of output species, if any.

\subsubsection{Example 1.}
\label{exam:addition}
\hspace{-1em} For the addition operation $g(s) = s_1 + s_2$with $s=(s_1, s_2)$ in $\mathbb{R}^{2}_{\geq 0}$, we can implement it by the MAS 
    \begin{gather}
    S_1 \stackrel{1}{\longrightarrow} S_1+G,~~~ 
    S_2 \stackrel{1}{\longrightarrow}S_2+G,~~~
    G\stackrel{1}{\longrightarrow}\emptyset,
       \notag
      \end{gather}
where $\mathcal{X}_1 = \{ S_1, S_2\}$ clearly act as \textit{catalyst species} (no concentration changes at all times, i.e., $\forall t,~s_1(t)=s_1(0),~s_2(t)=s_2(0)$). The dynamics follows species $\mathcal{X}_2 = \{G\}$, i.e., $\dot{g}(t)=s_1(t)+s_2(t)-g(t)$. Given $s_1(0)=s_1$ and $s_2(0)=s_2$, this yields the LSS of $G$ to be $\overline{g}=\lim_{t \rightarrow \infty}g(t)=s_1(0)+s_2(0)$, successfully storing the addition result $s_1+s_2$.





In the following, lowercase variables are used to represent the concentrations of the corresponding uppercase species.



\section{convergence rate lower bound for Biochemical Computational Module}\label{sec:convergence rate}
In this section, we define to specify the convergence rate when computing the vector-valued function $g(s) \in \mathbb{R}^{n_2}$, with $n_1+n_2 = n$, using the MAS (\ref{eq:bcm}) designed with multiple output species. 

\begin{defn}\label{def_conv_rate}
Suppose that (\ref{eq:bcm}) admits the limiting steady state $\bar{x} \in \mathbb{R}^n_{\geq 0}$ with $\bar{x}_2 = g(x_1(0)) \in \mathbb{R}^{n_2}$. The convergence rate of (\ref{eq:bcm}) to $\bar{x}$ is defined as
\begin{equation}\label{def_conv}
    \rho = -\limsup\limits_{t\rightarrow \infty} \frac{ \mathrm{ln} \Vert x(t) - \bar{x} \Vert}{t},
\end{equation}
whenever $\rho \in \left(0, \infty \right]$. Here, $\Vert \cdot \Vert$ denotes the $2$-norm of a vector. 
\end{defn}

In the context of molecular computation, the convergence rate implies that the biochemical computational module govern by (\ref{eq:bcm}) computes $g(s)$ at speed $\rho$. This definition extends the convergence rate for computing scalar functions using CRNs in \citep{anderson2025chemical} to the multiple output setting. The following lemma, adapted from \citep{anderson2025chemical}, further characterizes this rate.


\begin{lem}\label{lem_conv_rate}
Let the MAS (\ref{eq:bcm}) admits the LSS $\bar{x} \in \mathbb{R}^n_{\geq 0}$ with $\bar{x}_2 = g(x_1(0)) \in \mathbb{R}^{n_2}$. Suppose that there exist $M > 0, \rho > 0$ and $a > 0$ for all $t \geq 0$
\begin{equation}\label{lem5_1}
    \Vert x(t) - \bar{x} \Vert \leq M t^a e^{-\rho t},
\end{equation}
then (\ref{eq:bcm}) computes $ g(x_1(0)) = \bar{x}_2 $ at a rate that is at least $\rho$. Conversely, if the MAS (\ref{eq:bcm}) converges to $\bar{x}$ at a rate greater than $\rho$, then there exists a $M>0$ such that for all $t\geq 0$ 
\begin{equation}\label{eq:realization error}
    \Vert x(t) - \bar{x} \Vert \leq M e^{-\rho t}.
\end{equation}
\end{lem}
\begin{pf}
The forward direction follows directly by applying the operations in Definition \ref{def_conv_rate} to both sides of \eqref{lem5_1}. Conversely, for sufficiently large $T>0$, defining $M = \left(1 + \max_{t \in [0,T]}\{\|x(t) - \bar{x}\|\}\right)e^{\rho T}$ yields~(8) for all $t \geq 0$ due to the continuity of the norm.
\end{pf}

Definition \ref{def_conv_rate} establishes that the convergence rate is $\rho$ if the deviation between the trajectory of MAS and its LSS has the form of $f(t)e^{-\rho t}$ with $f(t)$ growing slower than any linear exponential function $e^{\epsilon t}$ ($\forall \epsilon >0$) when $t \to \infty$, i.e., satisfying $\lim_{t \to \infty} \frac{\mathrm{ln} f(t)}{t} = 0$ \citep{anderson2025chemical}. Lemma \ref{lem_conv_rate} further describes that $\rho$ serves as a lower bound for the convergence rate, derived from the dynamic characterization of the MAS. Conversely, when the explicit knowledge of the convergence rate is available, a more compact upper bound of the realization error can be found in (\ref{eq:realization error}), which certifies exponential convergence of the MAS and allows us to determine the finite time required to reach any prescribed computational accuracy $\epsilon$. Additionally, given $\epsilon>0$ and a finite operation time $T$, a CRN design is considered adequate if its convergence rate lower bound satisfies $\rho \geq \frac{1}{T} \ln(M/\epsilon)$. Therefore, linking $\rho$ to MAS characteristics facilitates the design of CRN structure and the rate constants tuning to meet accuracy. Note that the rate in (\ref{def_conv}) characterizes output convergence, as it lower-bounds the decay of both the entire system as well as the output species toward the LSSs.

Now we study the convergence rate lower bound of biochemical computational modules consisting only of inflow-outflow reactions and reactions involving one single non-catalytic reactant. Notably, the MAS govern by (\ref{eq:2}) and (\ref{mak}) becomes an affine linear system, i.e.,
\begin{equation}\label{eq:bcm_linear}
    \dot{x}(t) = A x(t) + u, x(0) \in \mathbb{R}^{n}_{\geq 0},
\end{equation}
where $A \in \mathbb{R}^{n \times n}, u \in \mathbb{R}^{n}$ represent the system matrix and inflow vector composed of rate constants. In the following, we provide a concrete convergence rate lower bound of implementing computations using (\ref{eq:bcm_linear}).
\begin{prop}\label{prop_linear_rate}
Suppose that the MAS (\ref{eq:bcm_linear}) admits the globally asymptotically stable equilibrium $\bar{x}= (\bar{x}_1, \bar{x}_2)^{\top} = A^{-1}u$, i.e., $A$ is Hurwitz with $n$ eigenvalues $\Re(\lambda_n) \leq \Re( \lambda_{n-1}) \leq \cdots \leq \Re (\lambda_1) < 0$, where $\lambda_1$ is the dominant eigenvalue with largest (least negative) real part,  then the convergence rate of calculating any function satisfying $g(x_1(0)) = \left[\mathbb{0}_{n_2 \times n_1} ~ I_{n_2 \times n_2}\right]A^{-1}u$ using (\ref{eq:bcm_linear}) is at least $\rho = \vert \Re (\lambda_1) \vert$.
\end{prop}
\begin{pf}
Considering the solution to (\ref{eq:bcm_linear}),
we have $\|x(t)+A^{-1}u\| \leq \|e^{At}\|\|x(0)+A^{-1}u\|$. For the  system matrix $A \in \mathbb{R}^{n \times n}$, let $A=PJP^{-1}$ be its Jordan form. 
Then, one can obtain
\begin{align*}
    \Vert e^{At}  \Vert & \leq \Vert P \Vert \Vert e^{Jt} \Vert \Vert P^{-1} \Vert. 
\end{align*}

Let $J = \Lambda + N$ be the decomposition into its diagonal part $\Lambda = \operatorname{diag}(\lambda_1, \dots, \lambda_n)$ and nilpotent part $N$. Notably, $\Lambda$ contains the eigenvalues of $A$, and $N^d = 0$ where $d = \max_i \{d_i\}$ is the order of the largest Jordan block. Therefore, combining the definition of $2$-norm for matrices leads to
\begin{align*}
\Vert e^{Jt} \Vert & =  \Vert e^{\Lambda t} e^{Nt} \Vert \leq \Vert  e^{\Lambda t} \Vert \Vert \sum^{d-1}_{j=0} \frac{N^j}{j!} t^j \Vert \\
    & \leq  \sqrt{\lambda_{\mathrm{max}} \left[ (e^{\Lambda t})^{\mathrm{H}} e^{\Lambda t}\right] } \sum^{d-1}_{j=0} 
 \Vert \frac{N^j}{j!} t^j  \Vert.
\end{align*}
\begin{align*}
 & = \sqrt{\lambda_{\mathrm{max}} \left[ (e^{\Lambda t})^{\mathrm{H}} e^{\Lambda t}\right] }  
 \sum^{d-1}_{j=0} \frac{t^j}{j!} \sqrt{\lambda_{\mathrm{max}} \left[ (N^j)^{\mathrm{H}} N^j \right] }.\\
\end{align*}
Due to the eigenvalues of $ (N^j)^{\mathrm{H}} N^j $ are composed of $\{0,1\}$, we have the estimation as follows
\begin{align*}\label{prop_deriv}
\Vert x(t) + A^{-1} u  \Vert  
& \leq C_1 \sqrt{\lambda_{\mathrm{max}} \left[ (e^{\Lambda t})^{\mathrm{H}} e^{\Lambda t}\right] }  
 \sum^{d-1}_{j=0} \frac{t^j}{j!} \\
 & = C_1 e^{ \Re (\lambda_1)} \sum^{d-1}_{j=0} \frac{t^j}{j!},
\end{align*}
where $C_1 = \Vert P \Vert \Vert P^{-1} \Vert \Vert x(0) + A^{-1} u \Vert$.
Combining (\ref{def_conv}) and Lemma \ref{lem_conv_rate}, it is demonstrated that 
\begin{equation*}
    \rho \geq -\Re(\lambda_1) = \vert \Re(\lambda_1) \vert
\end{equation*}
due to $\limsup\limits_{t\rightarrow \infty} \frac{\mathrm{ln}(\sum^{d-1}_{j=0}\frac{t^j}{j!})}{t} = 0$, i.e., the convergence rate of realizing computation using (\ref{eq:bcm_linear}) is at least $\vert \Re(\lambda _1) \vert$. $\hfill \Box$
\end{pf}

In the following, we extend our analysis to nonlinear MASs to evaluate more sophisticated biological computations, which necessitates a rigorous characterization of convergence rate for such positive polynomial systems. Suppose that the general MAS governed by (\ref{eq:2}) and (\ref{mak}) admits one LSS $\bar{x}$, and we denote $f(x)=\Gamma  \mathscr{K}(x) $. After applying the coordinate transformation of (\ref{eq:2}) and (\ref{mak}) and setting $\tilde{x} = x-\bar{x}$, one can have
\begin{equation}\label{eq:zero_bsm}
    \dot{\tilde{x}} = f(\tilde{x}+\bar{x}) \triangleq h(\tilde{x})
\end{equation}
with the zero LSS. Notably, it implies that there are no constant items in $h(\tilde{x})$. Due to the polynomial form of $h(\tilde{x})$, (\ref{eq:zero_bsm}) can be written as
\begin{equation}\label{linear_zero_bcm}
\dot{\tilde{x}} = \left.\frac{\partial h(\tilde{x})}{\partial \tilde{x}}\right|_{\tilde{x} = 0} \tilde{x} + \tilde{h}(\tilde{x}),
\end{equation}
where $\left.\frac{\partial h(\tilde{x})}{\partial \tilde{x}}\right|_{\tilde{x} = 0} = \left.\frac{\partial f(x)}{\partial x}\right|_{x = \bar{x}}$ is equal to the linear part of the system (\ref{eq:zero_bsm}). Therefore, $\tilde{h}(\tilde{x})$ consists solely of monomials of degree two or higher. We now formulate the following theorem to demonstrate that the nonlinear MASs,  served as the biochemical computational modules and admitting asymptotically stable equilibrium points, present the computation speed that can be quantitatively expressed as the dominant eigenvalue with largest real part of the Jacobian matrix, evaluated at the corresponding LSS. The proof relies on two technical lemmas provided in the Appendix, i.e., Lemma \ref{lem_scale}, which bounds the norm of higher-order polynomial terms in $\tilde{h}(\tilde{x})$, and Lemma \ref{lem_appro}, which shows that for any matrix over the complex field and $\epsilon > 0$, a diagonalizable matrix can be found to approximate it such that the corresponding eigenvalues differ by at most $\epsilon$  of the original ones. 


\begin{thm}\label{prop_nonlinear_rate}
Consider the MAS governed by (\ref{eq:2}) and (\ref{mak}) allowing an asymptotically stable equilibrium point $\bar{x}$. Suppose that the Jacobian matrix has $n$ eigenvalues satisfying $\Re{\lambda}_n \leq \Re{\lambda}_{n-1} \leq \cdots \leq \Re{\lambda}_1 < 0$. The convergence rate of calculating any function satisfying $g(x_1(0)) = \bar{x}_2$ using (\ref{eq:2}) and (\ref{mak}) is at least given by $\vert \Re{\lambda}_1 \vert$.
\end{thm}

\begin{pf}
Considering the system \eqref{linear_zero_bcm}, for $A = \left.\frac{\partial h(\tilde{x})}{\partial \tilde{x}}\right|_{\tilde{x} = 0}$ and any $\epsilon>0$, there exists a diagonalizable matrix $B_{\epsilon}$ and $ C_{\epsilon} = A-B_{\epsilon}$ satisfying $\Vert C_{\epsilon} \Vert < \epsilon$. Then, we have 
\begin{equation*}
    \begin{aligned}
\dot{\tilde{x}} =B_\epsilon \tilde{x} + C_\epsilon \tilde{x} + \tilde{h}(\tilde{x}). 
    \end{aligned}
\end{equation*}
Below, we can find an invertible matrix $P_{\epsilon}$ to perform a linear transformation on $ \tilde{x} $, i.e., $ z = P_{\epsilon}^{-1} \tilde{x},$ where 
$ P_{\epsilon}^{-1} \tilde{x} P_{\epsilon} = \Lambda_{\epsilon}$ and $\Lambda_{\epsilon}$ is a diagonal matrix. Then, we have
\begin{equation}\label{eq:tran_eq}
    \begin{aligned}
\dot{z} = P^{-1}_{\epsilon} \dot{\tilde{x}} = \Lambda_{\epsilon} z + P^{-1}_{\epsilon} C_{\epsilon} P_{\epsilon} z + P^{-1}_{\epsilon} \tilde{h} (P_{\epsilon} z).
    \end{aligned}
\end{equation}
Now we use the quadratic Lyapunov
function $V(z) = \frac{1}{2} \Vert z \Vert ^{2}$ to bound the energy of
deviations from equilibrium, and combined with \eqref{eq:tran_eq}, lemma \ref{lem_scale} and \ref{lem_appro}, we have
\begin{equation*}
    \begin{aligned}
        \dot{V} (z)  = & \frac{1}{2} (z^{\mathrm{H}} \dot{z} + \dot{z}^\mathrm{H}z ) \\
         = & z^{\mathrm{H}} \frac{\Lambda_{\epsilon} + \Lambda ^{\mathrm{H}}_{\epsilon}}{2}  z + z^{\mathrm{H}} \frac{ P^{-1}_{\epsilon} C_{\epsilon} P_{\epsilon}  + P^{\mathrm{H}}_{\epsilon} C^{\mathrm{H}}_{\epsilon} (P^{-1}_{\epsilon})^{\mathrm{H}} }{2} z \\
        & + \frac{z^{\mathrm{H}} P^{-1}_{\epsilon} \tilde{h} + \tilde{h}^{\mathrm{H}}(P^{-1}_{\epsilon})^{H}z  }{2} \\
     \leq & \lambda_{\mathrm{max}} \Re\left[\Lambda_{\epsilon}\right] \Vert z \Vert^{2} + c_1 \epsilon  \Vert z \Vert^2 + c_2 \Vert z \Vert ^{p+1} \\
     = & 2 ( \lambda_{\mathrm{max}} \Re\left[\Lambda_{\epsilon}\right] + c_1 \epsilon) V(z) +  \tilde{c}_2 V(z)^{\alpha},
    \end{aligned}
\end{equation*}
when $t > T$ with $T$ chosen large enough to ensure that $\Vert z\Vert < 1$. Here, $c_1 = \Vert P^{-1}_{\epsilon} \Vert \Vert P_{\epsilon} \Vert$, $c_2 = c \Vert P^{-1}_{\epsilon} \Vert$, $\tilde{c}_2 = 2^{\frac{p+1}{2}}c_2$ and $\alpha = \frac{p+1}{2} > 1$. Since the positive constant $\epsilon$ is arbitrary, one can have 
\begin{equation*}
    \begin{aligned}
        \dot{V} (z)  \leq 2  \lambda_{\mathrm{max}} \Re\left[\Lambda_{\epsilon}\right]  V(z) +  \tilde{c}_2 V(z)^{\alpha},
    \end{aligned}
\end{equation*}
Denoting $\lambda^{\prime}_{1}=\lambda_{\mathrm{max}} \Re\left[\Lambda_{\epsilon}\right]$, according to the Comparison Lemma and the solution to Bernoulli's equation, we have the inequality
\begin{equation*}
V(z(t)) \leq (-\frac{\tilde{c}}{2 \lambda^{\prime}_1} + (\frac{\tilde{c}}{2\lambda^{\prime}_{1}}+V(0)^{1-\alpha})e^{2\lambda^{\prime}_1 (1-\alpha)})^{\frac{1}{1-\alpha}}.
\end{equation*}
Based on these, we derive the upper bound of $\mathrm{ln} \Vert x - \bar{x} \Vert$ as follows
\begin{align*}
    \mathrm{ln} \Vert x-\bar{x} \Vert  & =  \mathrm{ln}  \Vert \tilde{x} \Vert \leq  \mathrm{ln} \Vert P_{\epsilon} \Vert + \frac{\mathrm{ln}2V(z)}{2}\\
    & \leq c_3 + \frac{\mathrm{ln} \left[-\frac{\tilde{c}_2}{2\lambda^{\prime}_1} + (\frac{\tilde{c}_2}{2\lambda^{\prime}_1}+V(0)^{1-\alpha})e^{2\lambda^{\prime}_1 (1-\alpha)} \right]  }{2(1-\alpha)}
\end{align*}
with $c_3 = \frac{\mathrm{ln2}}{2} + \mathrm{ln} \Vert P_{\epsilon} \Vert$. Therefore, considering the convergence rate (\ref{def_conv}), we have 
\begin{align*}
-\rho & \leq \limsup\limits_{t\rightarrow \infty}( \frac{\mathrm{ln} \left[-\frac{\tilde{c}_2}{2\lambda^{\prime}_1} + (\frac{\tilde{c}_2}{2\lambda^{\prime}_1}+V(0)^{1-\alpha})e^{2\lambda^{\prime}_1 (1-\alpha)t} \right]  }{2(1-\alpha)t}) \\
& = \limsup\limits_{t\rightarrow \infty} \frac{(\frac{\tilde{c}_2}{2\lambda^{\prime}_1}+V(0)^{1-\alpha})2\lambda^{\prime}_1 (1-\alpha)e^{2\lambda^{\prime}_1 (1-\alpha)t}}{2(1-\alpha)(-\frac{\tilde{c}_2}{2\lambda^{\prime}_1} + (\frac{\tilde{c}_2}{2\lambda^{\prime}_1}+V(0)^{1-\alpha})e^{2\lambda^{\prime}_1 (1-\alpha)t})}\\
& = \lambda^{\prime}_1.
\end{align*}
Recalling the lemma \ref{lem_appro}, for any $\epsilon > 0$ here, we can derive that $\vert \lambda^{\prime}_1 - \Re \lambda_1 \vert < \epsilon$. Therefore, it holds that $-\rho \leq \Re \lambda_1 + \epsilon$. Consequently, due to the  arbitrariness of $\epsilon$, it implies that the convergence rate of computing $g(x_1(0))=\bar{x}_2$ satisfying $\rho \geq - \Re \lambda_1 $, i.e., at least $\vert \Re \lambda_1 \vert$. $\hfill \Box$
\end{pf}

It should be noted that these characterizations are tailored for MASs where LSSs represent computational results, specifically Proposition \ref{prop_linear_rate} and Theorem \ref{prop_nonlinear_rate} provide the local convergence rate for individual equilibrium. Thus, they are applicable to multistable systems where distinct LSSs encode input-dependent computational results, potentially with varying convergence rate lower bounds. However, while CRNs exhibit diverse behaviors, limit cycles are excluded as they fail to yield the static convergent output required for computational tasks.

\section{Examples}\label{sec:example}
This section illustrates the theoretical lower bound through two numerical examples, demonstrating that it serves as the dominant factor governing computation speed, and system trajectories exhibiting consistent with our theoretical derivations.

\subsection{Hadamard Product}
This example illustrates Proposition \ref{prop_linear_rate} applied to the linear case. Considering the operation between vectors $a=(a_i), b=(b_i),c=(c_i) \in \mathbb{R}^2_{\geq 0}$, i.e., 
\begin{equation*}
a \odot b \odot c = 
\begin{pmatrix}
a_1 b_1 c_1 \\
a_2 b_2 c_2
\end{pmatrix} \triangleq 
\begin{pmatrix}
    z_1 \\
    z_2
\end{pmatrix},
\end{equation*}
then we construct the following biochemical reaction network governed by a linear system
\begin{equation}\label{eq:hadamard}
\begin{aligned}
    A_1 + B_1  \xrightarrow{k_1} A_1 + B_1 + M_1, & ~M_1 \xrightarrow{k_1} \varnothing, \\
    M_1 + C_1 \xrightarrow{k_2}  M_1 + C_1 + Z_1, & ~Z_1 \xrightarrow{k_2} \varnothing, \\
     A_2 + B_2  \xrightarrow{k_3} A_2 + B_2 + M_2, & ~M_2 \xrightarrow{k_3} \varnothing, \\
    M_2 + C_2 \xrightarrow{k_4}  M_2 + C_2 + Z_2, & ~Z_2 \xrightarrow{k_4} \varnothing 
\end{aligned}
\end{equation}
with input species $\{A_i, B_i, C_i\}$, and four output species $\{Z_i, M_i\}$ representing the resulted vector $z=(z_i) \in \mathbb{R}^2_{\geq 0}$. Eigenvalues of the system matrix of \eqref{eq:hadamard} is given by $\{k_1, k_2, k_3, k_4\}$.
According to Proposition \ref{prop_linear_rate}, we immediately find that the speed of computing $a \odot b \odot c$ using (\ref{eq:hadamard}) is at least $\rho_1 = \mathrm{min}\{k_1,k_2,k_3,k_4\}$, which corresponds to the smallest degradation rate among the involved species. This lower bound of convergence rate can be determined without explicitly solving the system \eqref{eq:hadamard}. Through one of the output species $Z_1$, Fig.\ref{fig:1a}(a) illustrates that this lower bound is the dominant factor determining the overall range of computation speed attainable using CRNs, and Fig. \ref{fig:1a}(b) shows the consistency between the true convergence behavior and our estimation of the convergence rate lower bound.

\begin{figure*}
  		\centering
  		\includegraphics[scale=0.87]{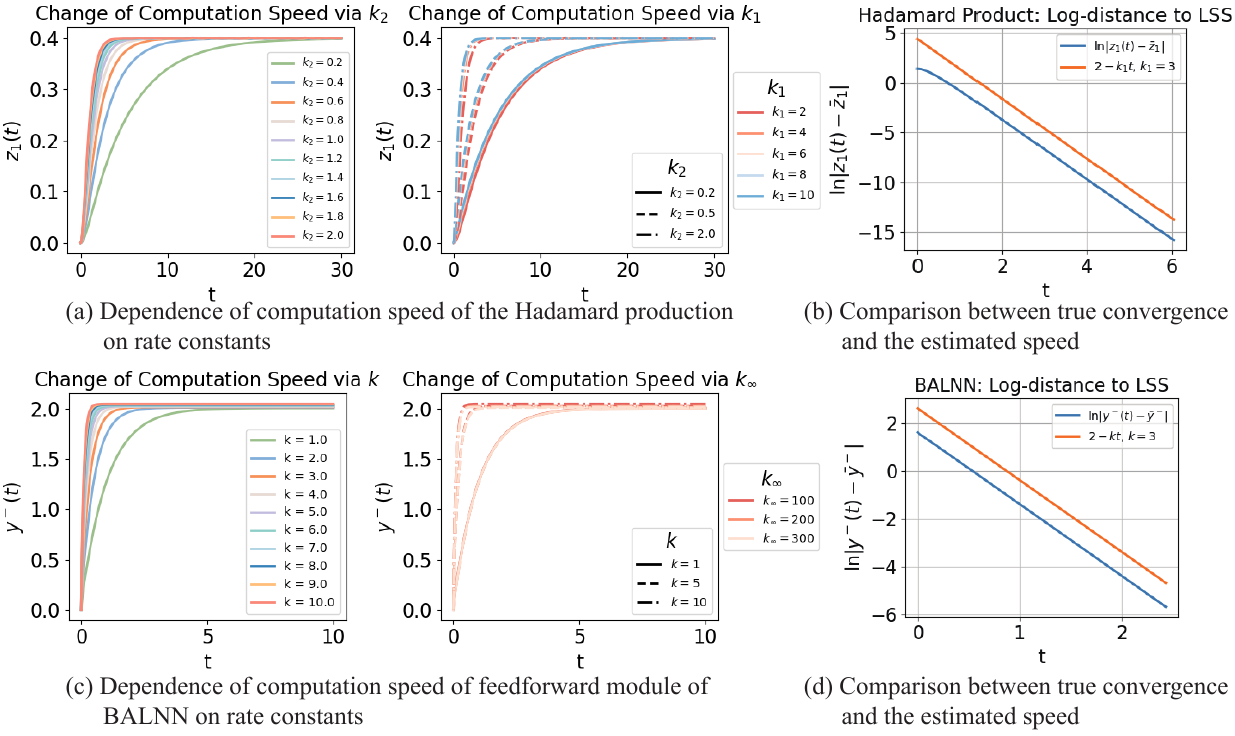}
  	 	\caption{Computation speed characterization for the Hadamard module and the BALNN feedforward module with $\rho_1 = k_2$ in (a) and $\rho_2 = k$ in (c). In (b) and (d), orange lines represent theoretical lower bounds $\rho = \vert \Re (\lambda_1) \vert$}, and the blue lines show the true convergence of \eqref{eq:hadamard} and \eqref{eq:balnn}. \label{fig:1a}
\end{figure*}

\subsection{Biochemical Adaptive Linear Neural Networks}
We validate Theorem \ref{prop_nonlinear_rate} for the nonlinear MAS by estimating the computation speed of the feedforward part in the biochemical adaptive linear neural network (BALNN) \citep{fan2022towards}, where the input sample $x = (x_1, \ldots, x_n) \in \mathbb{R}^{1 \times n}$ and weight vector $w = (w_1, \ldots, w_n) \in \mathbb{R}^{1 \times n}$.
The feedforward CRN with automatic computation capability to compute $y = x \cdot w$ is governed by a nonlinear MAS as follows
\begin{equation}
    \begin{split}
    	   \frac{dy^{+}(t)}{dt} & = k_1p^{+}(t)-ky^{+}(t)-k_{\infty}y^{+}(t)y^{-}(t), \\
    	   \frac{dy^{-}(t)}{dt} &= k_2p^{-}(t)-ky^{-}(t)-k_{\infty}y^{+}(t)y^{-}(t),
    	   \end{split}
    	   \label{eq:balnn}
\end{equation}
where $y^{\pm}(t)$ are the concentration of two output species of the feedforward part, and $p^{\pm}(t)$ maintain constant here and carry the positive and negative part of the linear weighted sum, respectively. Under the assumption $k_{\infty} \gg k$, $\bar{y}^{+}$ stores the positive output of the BALNN when $p^{+}(0) > p^{-}(0)$, whereas $\bar{y}^{-}$ stores its negative output. Eigenvalues of the Jacobian matrix are obtained as $\{-k, -k- \frac{-k_{\infty} \vert k_1 p^+(0) - k_2 p^-(0)\vert}{k}\}$, and therefore the convergence rate of the feedforward part in the BALNN is at least $ \rho_2 = k$. Through the output species $Y^-$, Fig. \ref{fig:1a}(c) illustrates how the degradation rate $k$ of $Y^+, Y^-$ influences the convergence rate compared with $k_{\infty}$, and Fig. \ref{fig:1a}(d) demonstrates the validity of our theoretical results.

\section{Conclusion}


This work develops a rigorous and concise characterization of convergence rate lower bound for biochemical computational modules, showing that the computation speed is fundamentally bounded by the dominant eigenvalues of the Jacobian matrix. Our results demonstrate how this characterization provides a theoretical metric for evaluating finite time computational accuracy and guides structure and parameter designs to achieve desired precision within a shorter time. However, current characterizations are inapplicable to singular Jacobians with zero eigenvalues, which may arise from conservation laws. Additionally, while maximizing the dominant eigenvalue provides one approach to accelerate biochemical molecular computation, transient performance is still influenced by the remaining spectrum. Validation against experimental data and addressing these theoretical limitations remain subjects for future research.




\bibliography{ifacconf}     

\appendix
 \section*{Appendix}
The following lemmas provide the technical foundations for the proof of Theorem \ref{prop_nonlinear_rate}.
\begin{lem}\label{lem_scale}
For one polynomial function $ h(x) \in \mathbb{C}^{n} $ with $x \in \mathbb{C}^n$, consisting of monomial items of degree $d \geq 2$, there exist $c >0$ and a fixed $p > 1$ such that $$\Vert h(x)  \Vert \leq  c \Vert x \Vert^{p}$$ when $\Vert x \Vert \leq 1$.    
\end{lem}
\begin{pf}
Without loss of generality, we focus on the $2$-norm here. Each scalar element of $h(x)$ is denoted as $h_j(x) = \sum^{m_j}_{k=1} a_{kj} \prod^{n}_{i=1} x^{b_{ij}}_{i}$
with $b_{ij}$ are nonnegative integers. Then, combining the Arithmetic Mean - Geometric Mean inequality, one can have 
\begin{align*}
\vert \prod^{n}_{i=1} x^{b_{ij}}_{i} \vert & \leq (\frac{\sum^{n}_{i=1} \vert x_i \vert^{b_{ij}}}{n})^n \leq (\frac{\sum^{n}_{i=1} \Vert x \Vert^{b_{ij}}}{n})^n\\
& \leq (\frac{n \Vert x \Vert^{b_{j}}}{n})^n = \Vert x \Vert^{nb_j}
\end{align*}
where $b_j = \mathrm{min}_{i}\{b_{ij}\}$ due to $\Vert x \Vert \leq 1$. Further, $\vert h_j(x) \vert \leq a_j \Vert x \Vert ^{nb_j}$ with $a_j = n \mathrm{max}_k \{a_{kj}\}$. With that, we can derive 
\begin{align*}
\Vert h(x) \Vert & = (\sum^{n}_{j=1}\vert h_j(x) \vert^2 )^{\frac{1}{2}} \leq (\sum^n_{j=1} a^2_j \Vert x \Vert ^{2 nb_j})^{\frac{1}{2}}\\
& \leq ( n a^2 \Vert x \Vert^{2nb})^{\frac{1}{2}} = \sqrt{n} \vert a \vert \Vert x \Vert ^{nb},
\end{align*}
where $a = \mathrm{max}_{j}\{a_j\}$ and $b = \mathrm{min}_{j} \{b_j\}$. Therefore, with $c=\sqrt{n} \vert a \vert$ and $p=nb$, it can be obtained that $p > 1$ since $b \geq 2$ when $n=1$ and $p > 1$ whenever $n \geq 2$.  $\hfill \Box$
\end{pf}


\begin{lem}\label{lem_appro}
Given $A \in \mathbb{C}^{n \times n}$, for every $\epsilon > 0$, there exists a diagonalizable matrix $B \in \mathbb{C}^{n \times n}$ such that $\Vert A-B\Vert < \epsilon$. Moreover, let $\lambda_i$ and $\lambda^{\prime}_i$ denote the eigenvalues of $A$ and $B$, respectively, we have $\vert \lambda_i - \lambda^{\prime}_i \vert < \epsilon$, $i = 1, \cdots, n$. 
\end{lem}

\begin{pf}
Considering the Schur decomposition, i.e., for $A \in \mathbb{C}^{n \times n}$, there exists a unitary matrix $U$ such that $U^{\mathrm{H}} A U = T$ where $T$ is an upper triangular matrix with the diagonal elements $\lambda_1, \cdots, \lambda_n$. For every $\epsilon > 0$, one can construct a perturbation matrix $D = \mathrm{diag} (d_1,\cdots,d_n)$ such that $\vert d_i \vert < \epsilon$ and eigenvalues of $T+D$ are distinct. It implies that there exists an invertible matrix $P$ such that 
\begin{equation*}
\begin{aligned}
   P^{-1} (T+D) P & = \mathrm{diag}(\lambda_1+d_1, \cdots, \lambda_n+d_n)\\
   & \triangleq \mathrm{diag} (\lambda^{\prime}_1, \cdots, \lambda^{\prime}_n) \triangleq \Lambda. 
\end{aligned}
\end{equation*}
Therefore, let $B =U (T+D) U^{\mathrm{H}}$, one can find an invertible matrix $UP$ satisfying $ (P^{-1} U^{\mathrm{H}})B (UP) =  \Lambda $, i.e., $B$ is diagonalizable.

Then, we estimate that 
\begin{equation*}
    \begin{aligned}
        \Vert B-A \Vert & = \Vert U D U^{\mathrm{H}} \Vert = \sqrt{\lambda_{\mathrm{max}} \left[ UD^{\mathrm{H}}DU^{\mathrm{H}} \right] } = \Vert D \Vert < \epsilon.
    \end{aligned}
\end{equation*}
Due to the construction, $\vert \lambda_i-\lambda^{\prime}_i \vert < \epsilon$ holds naturally. $\hfill \Box$
\end{pf}

\subsubsection{Remark 1.}
\hspace{-1em} Although perturbing eigenvalues may alter their ordering, for any $\epsilon > 0$, one can always choose perturbations such that the dominant eigenvalue index is preserved. Specifically, suppose that eigenvalues are ordered as $\Re \lambda_n \leq \Re \lambda_{n-1} \leq \cdots \leq  \Re \lambda_1$, and the first $k (k < n)$ ones share the same largest negative real part. To find the appropriate perturbations, considering $\lambda_i \in \{\lambda_1,\cdots,\lambda_k\}$, one can set $d_i = \frac{\epsilon_i}{2}e^{j\theta_i}$ with $\theta_1 \in \left[-\frac{\pi}{2}, \frac{ \pi}{2}\right]$, $\epsilon_1=\epsilon$, $\theta_i = \frac{\pi}{2} + \frac{(i-1)\pi}{k}$, $\epsilon_i = \mathrm{min}\{\epsilon, \frac{\Re \lambda_i - \Re \lambda_k}{2}\}$ for $i \in \{2,\cdots,k\}$, and $j$ denotes the imaginary unit in complex numbers here. As for $\lambda_i \in\{\lambda_{k+1},\cdots,\lambda_n\}$, it is only necessary to ensure that $ |d_i| < \min\{\epsilon, \frac{\Re \lambda_1 - \Re \lambda_i}{2}\} $ and that all perturbed eigenvalues remain distinct. 
\end{document}